\documentclass[12pt,reqno]{amsart}

\usepackage[utf8]{inputenc}
\usepackage{amssymb}
\usepackage{amsmath}
\usepackage{amsthm}
\usepackage{mathtools}
\usepackage{amsfonts}
\usepackage{bbm}
\usepackage{bookmark}
\usepackage{hyperref}
\hypersetup{pdfstartview={FitH}}
\usepackage{color}
\usepackage{xcolor}
\usepackage{mathrsfs}
\usepackage{enumitem}
\setlist[enumerate]{itemsep=0.7ex}
\usepackage[normalem]{ulem}
\usepackage{tikz}
\usepackage{tikz-3dplot}
\usetikzlibrary{trees,shapes,backgrounds}
\usepackage{caption}
\usetikzlibrary{arrows, arrows.meta, calc, positioning}
\usetikzlibrary{shapes.multipart}

\newcommand{\Rom}[1]{\uppercase\expandafter{\romannumeral #1}}

\hypersetup{unicode}
\hypersetup{breaklinks=true}

\theoremstyle{plain}
\newtheorem{thm}{Theorem}
\newtheorem{lemma}[thm]{Lemma}

\newtheorem{prop}[thm]{Proposition}
\newtheorem{cor}{Corollary}

\theoremstyle{definition}
\newtheorem{defn}{Definition}

\newtheorem{exampl}{Example}
\theoremstyle{remark}

\DeclareFontFamily{U}{mathx}{\hyphenchar\font45}
\DeclareFontShape{U}{mathx}{m}{n}{
<5> <6> <7> <8> <9> <10>
<10.95> <12> <14.4> <17.28> <20.74> <24.88>
mathx10
}{}
\DeclareSymbolFont{mathx}{U}{mathx}{m}{n}
\DeclareFontSubstitution{U}{mathx}{m}{n}
\DeclareMathAccent{\widecheck}{0}{mathx}{"71}

\title{A new class of function spaces generalizing the Arias-de-Reyna space}
\author{Jan~Moldav\v{c}uk}
\address{Department of Mathematical Analysis, Faculty of Mathematics and Physics, \newline Charles University, Sokolovsk\'a 83, 186 75 Praha 8, Czech Republic}
\email{jan.moldavcuk486@student.cuni.cz}
\date{\today}
\subjclass{46E30}
\keywords{Rearrangement-invariant quasi-Banach space, Lorentz space, Banach envelope, fundamental function}
\thanks{This research was supported by the grants 23-04720S and 26-21107S of the Czech Science Foundation.}

\begin{document}
\tdplotsetmaincoords{70}{110}

\begin{abstract}
    This paper studies the structure and properties of a rearrangement-invariant quasi-Banach space $QA_{\varphi,\psi}$ which generalizes the classical space $QA$ introduced by Arias-de-Reyna in connection with the study of the pointwise almost everywhere convergence of Fourier series. We present basic properties of $QA_{\varphi,\psi}$ and explore the relationship between $QA_{\varphi,\psi}$ and other rearrangement-invariant Banach spaces.
\end{abstract}

\maketitle

\section{Introduction}
One of the most fundamental problems in Fourier analysis is characterizing the space of all integrable functions on $[0,1]$ with respect to the Lebesgue measure $\lambda$, for which the Fourier series converges almost everywhere. Although this question has attracted attention for more than a century, it still remains unsolved to this day.

Kolmogorov's classical counterexample \cite{Kolmogoroff1923} states that there exists a function in $L^1$ with almost everywhere divergent Fourier series. On the other hand, Hunt \cite{MR238019} showed that for functions in $L^p$ with $p>1$, the Fourier series converges almost everywhere. This result has motivated mathematicians to search for function spaces between $L^1$ and $L^p$, where the convergence almost everywhere holds.

Throughout the years, there have been many results related to this problem. We refer to some of them \cite{MR1407066, MR1019791, MR241885}. In 1996, Antonov \cite{MR1407066} brought an important result that extended   many of these results, which states that in the space $L\log L\log\log\log L$ almost everywhere convergence holds.

Few years later, in 2002, Arias-de-Reyna \cite{MR1875141} defined a logconvex rearrangement-invariant quasi-Banach space $QA$, such that
\[
L\log L\log\log\log L\subset QA
\]
with the almost everywhere convergence property. This result gives a new and until now the best candidate for the biggest rearrangement-invariant space in which almost everywhere convergence holds. The deeper structure of this space was studied in the paper of Carro, Mastyło, and Rodríguez-Piazza \cite{MR2898729}.

We now give the definition of the space $QA$:
\begin{defn}
    We say that a measurable function $f$ on $[0,1]$ is in $QA$ if
    \begin{multline*}
    \|f\|_{QA}=\inf\left\{\sum_{n=1}^\infty (1+\log n)\|f_n\|_1\log\left(\frac{e\|f_n\|_\infty}{\|f_n\|_1}\right)\,:\right.\\
    \left.|f|\leq\displaystyle\sum_{n=1}^{\infty} f_n\;\;\text{a.e.},\,\{{f_n}\}_{n=1}^\infty\subset L^\infty,\, f_n\geq 0\right\}
    \end{multline*}
    is finite.
\end{defn}

In this paper, we introduce a more general family of spaces, denoted $QA_{\varphi,\psi}$, which contains $QA$ as a special case. Our aim is to study the basic structure of these spaces and to understand how they are related to classical rearrangement-invariant spaces, extending the results obtained in \cite{MR2898729}.
\begin{defn}\label{def1}
    Let $\psi:[0,\infty)\to [0,\infty)$ and $\varphi:[0,1]\to[0,\infty)$ be concave, non-decreasing functions that vanish only at the origin.
    The space $QA_{\varphi,\psi}$ consists of all measurable functions $f$ on $[0,1]$ for which
    \begin{multline*}
        \|f\|_{\varphi,\psi}=\inf\left\{\displaystyle\sum_{n=1}^\infty{\psi(n)\|f_n\|_\infty\varphi\left(\frac{\|f_n\|_1}{\|f_n\|_\infty}\right)}\,:\right.\\
    \left.|f|\leq\displaystyle\sum_{n=1}^{\infty} f_n\;\;\text{a.e.},\,\{{f_n}\}_{n=1}^\infty\subset L^\infty,\, f_n\geq 0\right\}
    \end{multline*}
    is finite (where $\frac{0}{0} \coloneqq 0$ by convention).
\end{defn}
First, observe that taking $\varphi(t)=t\log(e/t)$, $t\in(0,1]$ and $\psi(n)=1+\log n$, $n\in\mathbb{N}$, we get $QA_{\varphi,\psi}=QA$, so the introduction of $QA_{\varphi,\psi}$ may be seen as a meaningful generalization.
Since we are considering a wide class of function spaces, we lose information about logconvexity, but the rearrangement-invariant quasi-Banach property remains. This observation, together with other basic properties of spaces $QA_{\varphi,\psi}$, is the content of our first main result, Theorem 1. 

\begin{thm}\label{vlastnosti}
    \begin{enumerate}[label=(\roman*)]
        \item \label{r.i.} $QA_{\varphi,\psi}$ is an r.i.\ quasi-Banach space.
        \item \label{krajni}$L^\infty\hookrightarrow QA_{\varphi,\psi}\hookrightarrow L^1$.
        \item\label{density} The set of simple functions is dense in $QA_{\varphi,\psi}$.
        \item \label{L^infty}$QA_{\varphi,\psi}=L^\infty$ if and only if $\varphi(0_+)\neq 0$.
        \item\label{L^1} $QA_{\varphi,\psi}=L^1$ if and only if $\varphi\asymp \mathrm{id}$.
    \end{enumerate}
\end{thm}

The results \ref{r.i.} and \ref{density} were established in \cite{MR2898729} for the special case of the space $QA$. Here we show that they also hold in our general setting. We also prove the characterization of extreme cases in \ref{krajni} of the space $QA_{\varphi,\psi}$, namely \ref{L^infty} and \ref{L^1}.

To obtain deeper results concerning the structure of $QA_{\varphi,\psi}$ and its relation to Lorentz spaces, especially optimal embedding, we need to restrict ourselves to requirements for the growth of functions $\varphi$ and $\psi$.
 In the following, we assume that the cases \ref{L^infty} and \ref{L^1} in Theorem~\ref{vlastnosti} do not occur. We shall suppose that there exist constants $c\in(0,1)$, $C\in[1,\infty)$, and $p\in(0,1]$ such that
\begin{equation}\label{predpoklad}
    \frac{\varphi(t)}{t^c}\ \text{is non-decreasing on $(0,p]$},
\end{equation}
and for $n\in \mathbb{N}$,
\begin{equation}\label{predpoklad2}
    \psi(n^2)\leq C\psi(n).
\end{equation}
Under these assumptions, we introduce the function $\tau$, which plays a central role in characterizing the embeddings of Lorentz spaces into $QA_{\varphi,\psi}$. We define

\begin{equation}\label{tau}
    \tau(t)=\begin{cases}
            \varphi(t)\psi\left(\overline{\log}\,\left(\frac{\varphi(t)}{t}\right)\right)& \quad \text{if }\,t\in (0,1],\\
            0 & \quad \text{if }\,t=0,
        \end{cases}
\end{equation}
where $\overline{\log}\,t=1+\log^+t$, $t>0$. Now we are ready to state our second main result.

\begin{thm}\label{trida neobs}
    \begin{enumerate}[label=(\roman*)]
        \item \label{trida neobss} Let X be an r.i.\ Banach space with the fundamental function $\varphi_X$ such that 
        \[
        \displaystyle\lim_{t\to 0_+}\frac{\varphi_X(t)}{\tau(t)}=0,
        \]
        then $QA_{\varphi,\psi}$ does not contain $X$.
        \item \label{vnor} If there exists $t_0\leq1$ such that $\tau$ is a non-decreasing function on $[0,t_0]$, then
        \[
        \Lambda_{\widetilde{\tau}}\hookrightarrow QA_{\varphi,\psi},
        \]
        where $\widetilde{\tau}$ is concave function equivalent with the function $\tau$.
        \item \label{char} $QA_{\varphi,\psi}=\Lambda_\varphi$ if and only if $\psi\asymp1$ on $[1,\infty)$.
    \end{enumerate}
\end{thm}
This theorem provides the optimality of the Lorentz space generated by $\tau$ in the embedding into $QA_{\varphi,\psi}$ in the sense of the fundamental function. In particular, the result of Theorem \ref{trida neobs} gives a large class of rearrangement-invariant Banach spaces that are not contained in $QA_{\varphi,\psi}$. The special case of parts \ref{trida neobss} and \ref{vnor} of Theorem \ref{trida neobs} for $QA$ was established in \cite{MR2898729}, leading to the main result that the space $L\log L\log\log\log L$ is, in the sense of fundamental functions, the largest r.i.\ Banach space contained in $QA$.
As a consequence of the result \ref{char}, we obtain that $\psi\asymp1$ on $[1,\infty)$ is both a necessary and a sufficient condition for the normability of the space $QA_{\varphi,\psi}$. In addition, in Section \ref{section 4}, we show that the implication from right to left in the result \ref{char} is valid even without the additional assumptions (\ref{predpoklad}) and (\ref{predpoklad2}).

The paper is organized as follows. In Section 2, we recall some basic definitions and properties from the theory of Banach function spaces. In Section 3, we study the inner structure of the space $QA_{\varphi,\psi}$. In this section, we present parts \ref{r.i.}, \ref{krajni}, \ref{density} and the implication ($\Leftarrow$) of \ref{L^infty} and \ref{L^1} from the proof of Theorem \ref{vlastnosti}. In Section 4, we observe that the Banach envelope of the space $QA_{\varphi,\psi}$ is the Lorentz space $\Lambda_\varphi$. As a consequence of this observation, we obtain the rest of the proof of Theorem \ref{vlastnosti}, that is, ($\Rightarrow$) in \ref{L^infty} and \ref{L^1}. We also prove that there exists a large class of Lorentz spaces embedded in $QA_{\varphi,\psi}$.
In Section 5, we present the proof of Theorem \ref{trida neobs}.
We also give examples of the space $QA_{\varphi,\psi}$ and its corresponding function $\tau$ for various choices of the functions $\varphi$, $\psi$.

\section{Preliminaries}
In this chapter, we recall some basic definitions and results from the theory of Banach function spaces. Our main references are Bennett, Sharpley \cite[Chapters~\Rom{1},\Rom{2}]{MR928802} and \cite{MR2898729}.

Throughout this paper, all function spaces are considered over the interval $[0,1]$. We denote by $\lambda$ the Lebesgue measure on $[0,1]$, and by $\|\cdot\|_p$ the norm on $L^p= L^p([0,1],\lambda)$, for $1\leq p\leq \infty$. All equalities and inequalities between measurable functions are understood to hold almost everywhere (a.e.). If $f$ and $g$ are real-valued functions, then the symbol $f\asymp g$ means that there exist constants $C, c > 0$ such that $Cg\leq f\leq cg$, that is, $f$ and $g$ are equivalent.

Let $X$ be a real vector space. A functional $\|\cdot\|_X:X\to[0,\infty)$ is called a quasi-norm if it satisfies the following conditions:
\begin{enumerate}[label=(\arabic*)]
    \item $\|x\|_X=0$ if and only if $x=0$,
    \item $\|ax\|_X=|a|\|x\|_X$ for $a\in\mathbb{R}$, $x\in X$,
    \item there exists a constant $C\geq1$ such that $\|x+y\|_X\leq C(\|x\|_X+\|y\|_X)$ for $x,y\in X$.
\end{enumerate}
If condition (3) holds with $C = 1$, then $\|\cdot\|_X$ is called a norm. The couple $(X,\|\cdot\|_X)$ is called a (quasi-) normed space. If this space is complete, it is called a (quasi-) Banach space.

A (quasi-) Banach lattice $X$ (on $([0,1],\lambda)$) is a (quasi-) Banach space $X$ which is a subspace of measurable functions and satisfies:
\begin{enumerate}[label=(\arabic*)]
    \item there exists $u\in X$ with $u>0$ a.e.,
    \item if $|f|\leq|g|$ a.e., where $g\in X$ and $f$ is a measurable
    function, then $f\in X$ and $\|f\|_X\leq\|g\|_X$.
\end{enumerate}

Let $f$ be a measurable function. The distribution function $\lambda_f$ of the function $f$ is given by 
\[
\lambda_f(s)=\lambda\{x\in[0,1]:|f(x)|>s\},\qquad s\in[0,\infty).
\]
The decreasing rearrangement of $f$ is the function $f^\ast$ defined by 
\[
f^\ast(t)=\inf\{s\in[0,\infty):\lambda_f(s)\leq t\},\qquad t\in[0,\infty).
\]

A (quasi-) Banach lattice $X$ is called a rearrangement-invariant~(r.i.) (quasi-) Banach space $X$, if for all $f\in X$ and $g$ a measurable function such that $\lambda_f=\lambda_g$, it holds $g\in X$ and $\|f\|_X=\|g\|_X$.

The fundamental function of an r.i.\ (quasi-) Banach space is defined by $\varphi_X(t) =\|\chi_A\|_X$, where $A$ is a measurable set with $\lambda(A)=t$, for $t\in[0,1]$.

A non-decreasing function $\varphi:[0,1]\to[0,\infty)$ is called quasi-concave if $\varphi(0)=0$ and $\varphi(t)>0$ for all $t\in(0,1]$ and $\frac{\varphi(t)}{t}$ is non-increasing on $(0,1).$

We note that a quasi-concave function $\varphi$ is equivalent to its least concave majorant $\widetilde{\varphi}$, more precisely, $\widetilde{\varphi}\leq 2\varphi\leq 2\widetilde{\varphi}$ on $[0,1]$. If $\phi:[0,1]\to[0,\infty)$ is not quasi-concave but is equivalent to some concave, non-decreasing function that vanishes only at the origin, then we denote this function by $\widetilde{\phi}$.

Let $\varphi:[0,1]\to[0,\infty)$ be a concave, non-decreasing function that vanishes only at the origin. The Lorentz space $\Lambda_\varphi$ consists of all measurable functions $f$ such that 
\[
\|f\|_{\Lambda_\varphi}=\int_0^1 f^\ast(s)\mathrm{d}\varphi(s)<\infty.
\]

Since $\varphi$ is non-negative, non-decreasing and concave, then the Riemann-Stieltjes integral in the above definition can be rewritten in the form
\[
\|f\|_{\Lambda_\varphi}=\|f\|_\infty\varphi(0_+)+\int_0^1 f^\ast(s)\varphi'(s)\mathrm{d}s.
\]
The Lorentz space is an r.i.\ Banach space with a fundamental function equal to $\varphi$.

Let $X$ and $Y$ be (quasi-) normed spaces. We say that $X$ is embedded into $Y$, if $X\subset Y$ and the natural inclusion map of $X$ in $Y$ is continuous, that is, there exists a constant $C > 0$ such that
\[
\|x\|_Y\leq C\|x\|_X\,,\qquad x\in X.
\]
We write $X\hookrightarrow Y$. If both $X\hookrightarrow Y$ and $Y\hookrightarrow X$ hold, we write $X=Y$.

We note that if $\varphi:[0,1]\to[0,\infty)$ is a concave, non-decreasing function that vanishes only at the origin, then
\[
L^\infty\hookrightarrow\Lambda_\varphi\hookrightarrow L^1.
\]

We also mention that every simple function $s=\sum_{j=1}^n a_j\chi_{A_j}$ with $A_j$ pairwise disjoint sets and $a_1>a_2>\cdots>a_n>0$, can be rewritten in the form 
\begin{equation}\label{jedoducha}
    s=\sum_{k=1}^n b_k\chi_{B_k},
\end{equation}
where $b_k$ are positive and $B_1\subset B_2\subset\cdots\subset B_n$, more precisely, $b_k=a_k-a_{k+1}$ with $a_{n+1}=0$ and $B_k=\bigcup_{j=1}^k A_j$. Then
\[
\|s\|_{\Lambda_\varphi}=\sum_{k=1}^n b_k\varphi(\lambda(B_k)).
\]

\section{Basic Properties}
Unless stated otherwise, we will assume that the functions $\varphi$ and $\psi$ satisfy the conditions given in Definition \ref{def1}. 

In \cite{MR1875141} and \cite{MR2898729}, different definitions of the space $QA$ were used. The following lemma shows an equivalent definition of the space $QA_{\varphi,\psi}$, and as a consequence, we see that the two definitions of the space $QA$ used in the literature are equivalent. Throughout this paper, we will use both definitions of the space $QA_{\varphi,\psi}$ interchangeably, without mention.

\begin{lemma}
    Let $f$ be a measurable function. Then
    \begin{equation*}
        \|f\|_{\varphi,\psi}=\inf\left\{\displaystyle\sum_{n=1}^\infty{\psi(n)\|f_n\|_\infty\varphi\left(\frac{\|f_n\|_1}{\|f_n\|_\infty}\right)}\,:f=\displaystyle\sum_{n=1}^{\infty} f_n\;\;\text{a.e.},\{{f_n}\}_{n=1}^\infty\subset L^\infty\right\}.
    \end{equation*} 
\end{lemma}

\begin{proof}
    Let us denote 
    \begin{equation*}
        A_f=\inf\left\{\displaystyle\sum_{n=1}^\infty{\psi(n)\|f_n\|_\infty\varphi\left(\frac{\|f_n\|_1}{\|f_n\|_\infty}\right)}\,:f=\displaystyle\sum_{n=1}^{\infty} f_n\;\;\text{a.e.},\{{f_n}\}_{n=1}^\infty\subset L^\infty\right\}.
    \end{equation*}
    We prove $\|f\|_{\varphi,\psi}=A_f$ by showing both $\|f\|_{\varphi,\psi}\leq A_f$ and $A_f\leq \|f\|_{\varphi,\psi}$. 
    Let us start with $\|f\|_{\varphi,\psi}\leq A_f$. If $A_f=\infty$, there is nothing to prove. Let $f$ be a measurable function for which $A_f<\infty$. Let $\varepsilon>0$, we choose $\{f_n\}_{n=1}^\infty\subset L^\infty$ such that $f=\sum_{n=1}^\infty f_n$ and
    \[
    A_f\leq \displaystyle\sum_{n=1}^\infty{\psi(n)\|f_n\|_\infty\varphi\left(\frac{\|f_n\|_1}{\|f_n\|_\infty}\right)}<A_f+\varepsilon.
    \]
    Let $g_n=|f_n|$, then $|f|=\left|\sum_{n=1}^\infty f_n\right|\leq \sum_{n=1}^\infty g_n$, $g_n\geq 0$ and 
    \[
        \|f\|_{\varphi,\psi}\leq \displaystyle\sum_{n=1}^\infty \psi(n)\|g_n\|_\infty\varphi\left(\frac{\|g_n\|_1}{\|g_n\|_\infty}\right)=\displaystyle\sum_{n=1}^\infty{\psi(n)\|f_n\|_\infty\varphi\left(\frac{\|f_n\|_1}{\|f_n\|_\infty}\right)}< A_f+\varepsilon.
    \]
    By letting $\varepsilon \to 0$, we obtain the inequality.

    On the other hand, we shall prove $A_f\leq\|f\|_{\varphi,\psi}$. We can assume $f\in QA_{\varphi,\psi}$, since otherwise there is nothing to prove. Given $\varepsilon>0$, we choose $\{f_n\}_{n=1}^\infty\subset L^\infty,\,f_n\geq0$ such that $|f|\leq\sum_{n=1}^\infty f_n$ and 
    \[
    \|f\|_{\varphi,\psi}\leq\displaystyle\sum_{n=1}^\infty \psi(n)\|f_n\|_\infty\varphi\left(\frac{\|f_n\|_1}{\|f_n\|_\infty}\right)< \|f\|_{\varphi,\psi} +\varepsilon.
    \]
    From the concavity of $\varphi$, it follows that $\frac{\varphi(t)}{t}$ is non-increasing on $(0,1]$, as a consequence of it, $x\varphi\left(\frac{a}{x}\right)$ is non-decreasing on $[a,\infty)$. 
    We can assume that $|f|=\sum_{n=1}^\infty f_n$, because for given $g,h\in L^\infty,\,0\leq g\leq h$, we have that $x\varphi(\frac{\|g\|_1}{x})$ is non-decreasing on $[\|g\|_1,\infty)$, and therefore
    \begin{equation}\label{usporadani}
        \|g\|_{\infty}\varphi\left(\frac{\|g\|_1}{\|g\|_{\infty}}\right)\leq 
        \|h\|_{\infty}\varphi\left(\frac{\|g\|_1}{\|h\|_{\infty}}\right)\leq
        \|h\|_{\infty}\varphi\left(\frac{\|h\|_1}{\|h\|_{\infty}}\right).
    \end{equation}
    Then
    \[
    |f|=\displaystyle\sum_{n=1}^\infty f_n=\displaystyle\sum_{n=1}^\infty (f_n|_{\{f\geq0\}}+f_n|_{\{f<0\}})=\displaystyle\sum_{n=1}^\infty f_n|_{\{f\geq0\}}\,+\,\displaystyle\sum_{n=1}^\infty f_n|_{\{f<0\}}=f^++f^-.
    \]
    Therefore
    \[
    f=f^+-f^-=\displaystyle\sum_{n=1}^\infty f_n|_{\{f\geq0\}}\,-\,\displaystyle\sum_{n=1}^\infty f_n|_{\{f<0\}}=\displaystyle\sum_{n=1}^\infty (f_n|_{\{f\geq0\}}-f_n|_{\{f<0\}}).
    \]
    Let us define $g_n=(f_n|_{\{f\geq0\}}-f_n|_{\{f<0\}})$, then $g_n\in L^\infty,\,f=\sum_{n=1}^\infty g_n$ and $|g_n|=f_n$, because $f_n\geq0$. We obtain 
    \[
        A_f\leq \displaystyle\sum_{n=1}^\infty \psi(n)\|g_n\|_\infty\varphi\left(\frac{\|g_n\|_1}{\|g_n\|_\infty}\right)=\displaystyle\sum_{n=1}^\infty{\psi(n)\|f_n\|_\infty\varphi\left(\frac{\|f_n\|_1}{\|f_n\|_\infty}\right)}
        < \|f\|_{\varphi,\psi}+\varepsilon.
    \]
    Letting $\varepsilon\to 0$, we obtain the desired inequality. 
\end{proof}

\begin{prop}\label{banach}
    $QA_{\varphi,\psi}$ is a quasi-Banach lattice.
\end{prop}

\begin{proof}
    It is clear that $\|f\|_{\varphi,\psi}\geq 0$ for $f\in QA_{\varphi,\psi}$, and if $f=0$, then $\|f\|_{\varphi,\psi}= 0$ by choosing $f_n\equiv0$ for all $n\in \mathbb{N}$. If $\|f\|_{\varphi,\psi}= 0$, then for given $\varepsilon>0$, we find $\{f_n\}_{n=1}^\infty\subset L^\infty,\,f_n\geq0$ such that $|f|\leq\sum_{n=1}^\infty f_n$ and 
    \[
    0\leq\displaystyle\sum_{n=1}^\infty \psi(n)\|f_n\|_\infty\varphi\left(\frac{\|f_n\|_1}{\|f_n\|_\infty}\right)<\varepsilon.
    \]
    Applying the Monotone Convergence Theorem and using the concavity of $\varphi$, we obtain
    \begin{align*}
        \frac{\varepsilon}{\varphi(1)\psi(1)}&> \frac{1}{\varphi(1)\psi(1)}\displaystyle\sum_{n=1}^\infty \psi(n)\|f_n\|_\infty\varphi\left(\frac{\|f_n\|_1}{\|f_n\|_\infty}\right)\geq \frac{1}{\varphi(1)}\displaystyle\sum_{n=1}^\infty \|f_n\|_\infty\varphi\left(\frac{\|f_n\|_1}{\|f_n\|_\infty}\right)
        \\
        &\geq\displaystyle\sum_{n=1}^\infty\|f_n\|_\infty\frac{\|f_n\|_1}{\|f_n\|_\infty}= \left\|\displaystyle\sum_{n=1}^\infty f_n\right\|_1\geq \|f\|_1\geq 0.
    \end{align*}
    Since $\varepsilon$ was arbitrary, we conclude $\|f\|_1=0$, which implies $f=0$.

    The equality $\|af\|_{\varphi,\psi}= |a|\|f\|_{\varphi,\psi}$ follows directly from the definition.
    
    Assume that $f=\sum_{i=1}^n f_i,\,f_i\in QA_{\varphi,\psi},\,n\in\mathbb{N}\cup\{\infty\}$. Given $\varepsilon>0$, for each $i\in\{1,\dots,n\}$ if $n\in\mathbb{N}$, or $i\in\mathbb{N}$ if $n=\infty$, we find $\{f_{i,j}\}_{j=1}^\infty\subset L^\infty$ such that $|f_i|\leq\sum_{j=1}^\infty f_{i,j},\,f_{i,j}\geq0$ and
    \[
    \|f_i\|_{\varphi,\psi}\leq \displaystyle\sum_{j=1}^\infty{\psi(j)\|f_{i,j}\|_\infty\varphi\left(\frac{\|f_{i,j}\|_1}{\|f_{i,j}\|_\infty}\right)}<\begin{cases} \|f_i\|_{\varphi,\psi}+\frac{\varepsilon}{2^{2i-n}(2^n-1)}&\quad \text{if }\,n\in\mathbb{N},\\
    \|f_i\|_{\varphi,\psi}+\frac{\varepsilon}{2^{2i}}&\quad \text{if }\,n=\infty.
    \end{cases}
    \]
    Then $|f|\leq\sum_{i=1}^n |f_i|\leq\sum_{i=1}^n\sum_{j=1}^\infty f_{i,j}$, and thus
    \[
    \|f\|_{\varphi,\psi}\leq \displaystyle\sum_{i=1}^n\displaystyle\sum_{j=1}^\infty{\psi(\sigma(i,j))\|f_{i,j}\|_\infty\varphi\left(\frac{\|f_{i,j}\|_1}{\|f_{i,j}\|_\infty}\right)},
    \]
    where 
    \[
    \sigma:\begin{cases}
        \{1,\dots,n\}\times\mathbb{N}\to\mathbb{N}&\quad \text{if }\,n\in\mathbb{N},\\
        \mathbb{N}\times\mathbb{N}\to\mathbb{N}&\quad \text{if }\,n=\infty,
    \end{cases}
    \]
    is a bijection. We can take the shifted pairing function
    \[
    \sigma(i,j)=
    \begin{cases}
        2^{i-1}(2j-1)& \quad \text{if }\, i<n,\\
        2^{i-1}j& \quad \text{if }\, i=n.
    \end{cases}
    \]
    By the subadditivity of the concave function $\psi$, we have
    \[
    \psi(\sigma(i,j))\leq \psi(2^{i-1}(2j-1))\leq 2^{i-1}\psi(2j-1)\leq 2^{i-1}\psi(2j)\leq 2^i\psi(j).
    \]
    Therefore
    \begin{align*}
        \|f\|_{\varphi,\psi}&\leq \displaystyle\sum_{i=1}^n \displaystyle\sum_{j=1}^\infty{\psi(\sigma(i,j))\|f_{i,j}\|_\infty\varphi\left(\frac{\|f_{i,j}\|_1}{\|f_{i,j}\|_\infty}\right)}\\
        &\leq \displaystyle\sum_{i=1}^n 2^i\displaystyle\sum_{j=1}^\infty{\psi(j)\|f_{i,j}\|_\infty\varphi\left(\frac{\|f_{i,j}\|_1}{\|f_{i,j}\|_\infty}\right)}\\
        &< \displaystyle\sum_{i=1}^n {2^i\|f_i\|_{\varphi,\psi}}+\varepsilon.
    \end{align*}
    By letting $\varepsilon\to0$, for $n=2$ we obtain the quasi-triangle inequality
    \begin{equation}\label{quasi-norm}
        \|f_1+f_2\|_{\varphi,\psi}\leq 4(\|f_1\|_{\varphi,\psi}+\|f_2\|_{\varphi,\psi}),\qquad f_1,f_2\in QA_{\varphi,\psi}.
    \end{equation}

    Let $\{f_n\}_{n=1}^\infty\subset QA_{\varphi,\psi}$ be a Cauchy sequence in $QA_{\varphi,\psi}$. We choose a subsequence $\{f_{n_k}\}_{k=1}^\infty$ such that $\|f_{n_{k+1}}-f_{n_k}\|_{\varphi,\psi}\leq 4^{-k}$ for all $k\in\mathbb{N}$. Then due to the above inequality for $n=\infty$, we have
    \[
    \left\|\displaystyle\sum_{k=1}^\infty (f_{n_{k+1}}-f_{n_{k}})\right\|_{\varphi,\psi}\leq \displaystyle\sum_{k=1}^\infty 2^k\|f_{n_{k+1}}-f_{n_k}\|_{\varphi,\psi}\leq \displaystyle\sum_{k=1}^\infty \frac{1}{2^k}<\infty.
    \]
    Therefore $f_{n_j}=\sum_{k=1}^{j-1} (f_{n_{k+1}}-f_{n_{k}})+f_{n_1}$ converges in $QA_{\varphi,\psi}$, which implies that $f_n$ converges in $QA_{\varphi,\psi}$, and thus $QA_{\varphi,\psi}$ is a complete space.

    The first lattice condition is trivially satisfied because $L^\infty\subset QA_{\varphi,\psi}$. Let $g\in QA_{\varphi,\psi}$ and $f$ be a measurable function such that $|f|\leq|g|$. Let $\{g_n\}_{n=1}^\infty\subset L^\infty$, $g_n\geq 0$ such that $|f|\leq|g|\leq\sum_{n=1}^\infty g_n$ and 
    \[
    \displaystyle\sum_{n=1}^\infty \psi(n)\|g_n\|_\infty\varphi\left(\frac{\|g_n\|_1}{\|g_n\|_\infty}\right)<\infty.
    \]
    Then clearly $f\in QA_{\varphi,\psi}$ and $\|f\|_{\varphi,\psi}\leq\|g\|_{\varphi,\psi}$, which proves the lattice property.   
\end{proof}

In the following, by using Proposition \ref{banach} we establish most of the parts of Theorem \ref{vlastnosti}, specifically \ref{r.i.}, \ref{krajni}, \ref{density}, and one implication ($\Leftarrow$) in \ref{L^infty} and \ref{L^1}. The remainder of the proof of Theorem \ref{vlastnosti} will be presented in Section \ref{section 4}.
\begin{proof}[Proof of Theorem \ref*{vlastnosti}, first part]
    (ii) Let $f\in L^\infty$, then clearly
    \begin{equation}\label{vnoreni infty}
        \|f\|_{\varphi,\psi}\leq \psi(1)\|f\|_\infty\varphi\left(\frac{\|f\|_1}{\|f\|_\infty}\right)\leq \varphi(1)\psi(1)\|f\|_\infty,
    \end{equation}
    which proves that $L^\infty\hookrightarrow QA_{\varphi,\psi}$.

    Let $f\in QA_{\varphi,\psi}$ and $\varepsilon>0$. We find $\{f_n\}_{n=1}^\infty\subset L^\infty,\,f_n\geq0$ such that $|f|\leq\sum_{n=1}^\infty f_n$ and 
    \[
    \|f\|_{\varphi,\psi}\leq \displaystyle\sum_{n=1}^\infty \psi(n)\|f_n\|_\infty\varphi\left(\frac{\|f_n\|_1}{\|f_n\|_\infty}\right)< \|f\|_{\varphi,\psi}+\varepsilon.
    \]
    Applying the Monotone Convergence Theorem and using the concavity of $\varphi$, we
    obtain
    \begin{align*}
        \|f\|_{\varphi,\psi}+\varepsilon&> \displaystyle\sum_{n=1}^\infty \psi(n)\|f_n\|_\infty\varphi\left(\frac{\|f_n\|_1}{\|f_n\|_\infty}\right)\geq \psi(1)\displaystyle\sum_{n=1}^\infty\|f_n\|_\infty\frac{\|f_n\|_1}{\|f_n\|_\infty}\varphi(1)\\
        &= \varphi(1)\psi(1)\left\|\displaystyle\sum_{n=1}^\infty f_n\right\|_1\geq \varphi(1)\psi(1)\|f\|_1.
    \end{align*}
    Since $\varepsilon$ was arbitrary, we obtain the desired continuous inclusion
    \begin{equation}\label{vnoreni 1}
        \|f\|_1\leq \frac{1}{\varphi(1)\psi(1)}\|f\|_{\varphi,\psi},\qquad f\in QA_{\varphi,\psi}.
    \end{equation}

\ref*{r.i.}
    The fact that $QA_{\varphi,\psi}$ is a quasi-Banach lattice was proven in Proposition \ref{banach}. Let $f\in QA_{\varphi,\psi}$ and $g$ be a measurable function such that $\lambda_f=\lambda_g$. By the Calderón Theorem \cite[\Rom{3}.Theorem~2.11]{MR928802} there exists a linear operator $T:L^1\to L^1$ such that $T|_{L^\infty}:L^\infty\to L^\infty$,
    \begin{equation*}
        \begin{aligned}
        \|T\|_{L^1\to L^1}&=\sup\{\|Tx\|_1:\,\|x\|_1\leq 1\}\leq 1,\\
        \|T\|_{L^\infty\to L^\infty}&=\sup\{\|Tx\|_\infty:\,\|x\|_\infty\leq 1\}\leq 1
        \end{aligned}
    \end{equation*}
    and $g=Tf$.

    Let $\varepsilon>0$, we find $\{f_n\}_{n=1}^\infty$ such that $f=\sum_{n=1}^\infty f_n$ and 
    \[
    \|f\|_{\varphi,\psi}\leq\displaystyle\sum_{n=1}^\infty \psi(n)\|f_n\|_\infty\varphi\left(\frac{\|f_n\|_1}{\|f_n\|_\infty}\right) < \|f\|_{\varphi,\psi}+\varepsilon.
    \]
    Then $T(\sum_{n=1}^k f_n)=\sum_{n=1}^kTf_n$, for $k\in \mathbb{N}$ and $Tf_n\in L^\infty$. Analogously to \eqref{usporadani}, since $x\varphi\left(\frac{\|Tf_n\|_1}{x}\right)$ is a non-decreasing function on $[\|Tf_n\|_1,\infty)$ for every $n\in\mathbb{N}$, and using the inequalities $\|Tf_n\|_1\leq\|Tf_n\|_\infty\leq\|f_n\|_\infty$ and $\|Tf_n\|_1\leq\|f_n\|_1$, we obtain
    \[
    \|Tf_n\|_\infty\varphi\left(\frac{\|Tf_n\|_1}{\|Tf_n\|_\infty}\right)\leq \|f_n\|_\infty\varphi\left(\frac{\|f_n\|_1}{\|f_n\|_\infty}\right). 
    \]
     Therefore
    \begin{align*}
        \left\|\sum_{n=k+1}^\infty Tf_n\right\|_{\varphi,\psi}&\leq \sum_{n=k+1}^\infty \psi(n)\|Tf_n\|_\infty\varphi\left(\frac{\|Tf_n\|_1}{\|Tf_n\|_\infty}\right)\\
        &\leq \displaystyle\sum_{n=k+1}^\infty \psi(n)\|f_n\|_\infty\varphi\left(\frac{\|f_n\|_1}{\|f_n\|_\infty}\right)\to 0,
    \end{align*}
    and 
    \[
    \left\|f-\sum_{n=1}^k f_n\right\|_{\varphi,\psi}\leq \sum_{n=k+1}^\infty \psi(n)\|f_n\|_\infty\varphi\left(\frac{\|f_n\|_1}{\|f_n\|_\infty}\right)\to 0.
    \]
    Hence by (\ref{vnoreni 1}) and combining the above inequalities,
    \begin{align*}
        \left\|g-\sum_{n=1}^\infty Tf_n\right\|_1&\leq \left\|Tf-\sum_{n=1}^k Tf_n\right\|_1+\left\|\sum_{n=k+1}^\infty Tf_n\right\|_1\\
        &\leq \left\|f-\sum_{n=1}^k f_n\right\|_1+\left\|\sum_{n=k+1}^\infty Tf_n\right\|_1\\
        &\leq \frac{1}{\varphi(1)\psi(1)}\left(\left\|f-\sum_{n=1}^k f_n\right\|_{\varphi,\psi}+\left\|\sum_{n=k+1}^\infty Tf_n\right\|_{\varphi,\psi}\right)\to 0.
    \end{align*}
    Thus $g=\sum_{n=1}^\infty Tf_n$ and we conclude
    \begin{align*}
        \|g\|_{\varphi,\psi}&\leq \sum_{n=1}^\infty \psi(n)\|Tf_n\|_\infty\varphi\left(\frac{\|Tf_n\|_1}{\|Tf_n\|_\infty}\right)\\
        &\leq \displaystyle\sum_{n=1}^\infty \psi(n)\|f_n\|_\infty\varphi\left(\frac{\|f_n\|_1}{\|f_n\|_\infty}\right)<\|f\|_{\varphi,\psi}+\varepsilon.
    \end{align*}
    This yields $g\in QA_{\varphi,\psi}$, and since $\varepsilon>0$ was arbitrary, we obtain 
    \[
    \|g\|_{\varphi,\psi}\leq \|f\|_{\varphi,\psi}.
    \]
    Applying Calderón's Theorem in the opposite direction (to the function $g$), we get
    \[
    \|f\|_{\varphi,\psi}=\|g\|_{\varphi,\psi},
    \]
    and this shows that $QA_{\varphi,\psi}$ is an r.i.\ quasi-Banach space.

    (iii) Let $f\in QA_{\varphi,\psi}$. We choose $\{f_n\}_{n=1}^\infty\subset L^\infty$ such that $f=\sum_{n=1}^\infty f_n$ and 
    \[
    \displaystyle\sum_{n=1}^\infty \psi(n)\|f_n\|_\infty\varphi\left(\frac{\|f_n\|_1}{\|f_n\|_\infty}\right)< \infty.
    \]
    Let us define $g_n=\sum_{j=1}^n f_j\in L^\infty$. Then 
    \begin{equation*}
        \|f-g_n\|_{\varphi,\psi}\leq \displaystyle\sum_{j=n+1}^\infty \psi(j)\|f_j\|_\infty\varphi\left(\frac{\|f_j\|_1}{\|f_j\|_\infty}\right)\to0.
    \end{equation*}
    Let $\varepsilon>0$, we find $n_0\in\mathbb{N}$ such that $\|f-g_{n_0}\|_{\varphi,\psi}<\varepsilon$.
    Since the set of simple functions is dense in $L^\infty$, we find a simple function $s$ such that $\|g_{n_0}-s\|_\infty<\varepsilon$. Then by  (\ref{quasi-norm}) and (\ref{vnoreni infty}),
    \[
    \|f-s\|_{\varphi,\psi}\leq 4(\|f-g_{n_0}\|_{\varphi,\psi}+\varphi(1)\psi(1)\|g_{n_0}-s\|_\infty)< 4(1+\varphi(1)\psi(1))\varepsilon.
    \]
    Hence the set of simple functions is dense in $QA_{\varphi,\psi}$.
    
    (iv) ($\Leftarrow$) It is enough to prove that $QA_{\varphi,\psi}\hookrightarrow L^\infty$. Since $\varphi(0_+)\neq 0$, we have $\varphi\asymp1$ on $(0,1]$, so we can assume that 
    \[
    \varphi(t)=\begin{cases}
        1& \quad \text{if }\, t\in(0,1],\\
        0& \quad \text{if }\, t=0.
    \end{cases}
    \]
    Let $f\in QA_{\varphi,\psi}$, we choose $\{f_n\}_{n=1}^\infty\subset L^\infty$ such that $f=\sum_{n=1}^\infty f_n$ and 
    \[
    \displaystyle\sum_{n=1}^\infty \psi(n)\|f_n\|_\infty<\infty.
    \]
    Then 
    \[
    \|f\|_\infty\leq \displaystyle\sum_{n=1}^\infty \|f_n\|_\infty\leq \frac{1}{\psi(1)}\displaystyle\sum_{n=1}^\infty \psi(n)\|f_n\|_\infty,
    \]
    therefore by taking infimum over both sides, we get $QA_{\varphi,\psi}\hookrightarrow L^\infty$.
    
    (v) ($\Leftarrow$) We can assume that $\varphi(t)=t$, $t\in [0,1]$.
    First we shall prove that $\|\cdot\|_{\varphi,\psi}$ is a norm. Let $f,g\in QA_{\varphi,\psi}$ and $\{f_n\}_{n=1}^\infty,\{g_n\}_{n=1}^\infty\subset L^\infty,\,f_n,g_n\geq0$ such that $|f|\leq\sum_{n=1}^\infty f_n$ and $|g|\leq\sum_{n=1}^\infty g_n$. Then 
    $|f+g|\leq|f|+|g| \leq \sum_{n=1}^\infty(f_n+g_n)$, and thus
    \[ \|f+g\|_{\varphi,\psi}\leq\displaystyle\sum_{n=1}^\infty\psi(n)\|f_n+g_n\|_1\leq \displaystyle\sum_{n=1}^\infty\psi(n)\|f_n\|_1+\displaystyle\sum_{n=1}^\infty\psi(n)\|g_n\|_1.
    \]
    Since $\{f_n\}_{n=1}^\infty$ and $\{g_n\}_{n=1}^\infty$ were arbitrary, we obtain that $\|\cdot\|_{\varphi,\psi}$ is a norm.

    Let $s=\sum_{i=1}^n a_i\chi_{A_i}$ be a simple function on $[0,1]$, where the sets $A_j$ are pairwise disjoint. Then 
    \[
    \|s\|_{\varphi,\psi}\leq \displaystyle\sum_{i=1}^n \|a_i\chi_{A_i}\|_{\varphi,\psi}\leq \displaystyle\sum_{i=1}^n \psi(1)\|a_i\chi_{A_i}\|_1=\psi(1)\|s\|_1.
    \]
    By part \ref{density}, the set of simple functions is dense in $QA_{\varphi,\psi}$, thus for all $\varepsilon>0$ and $f\in QA_{\varphi,\psi}$ we can find a simple function $s$ such that $\|f-s\|_{\varphi,\psi}<\varepsilon$. Then by (\ref{vnoreni 1}) also $\psi(1)\|f-s\|_1<\varepsilon$, and thus by (\ref{vnoreni 1}) and triangle inequality 
    \[
        \psi(1)\|f\|_1\leq \|f\|_{\varphi,\psi}\leq \|s\|_{\varphi,\psi}+\varepsilon\leq \psi(1)\|s\|_1+\varepsilon\leq \psi(1)
        \|f\|_1+2\varepsilon.
    \]
    Since $\varepsilon$ was arbitrary, we have 
    \[
    \psi(1)\|f\|_1\leq \|f\|_{\varphi,\psi}\leq \psi(1)\|f\|_1,\qquad f\in QA_{\varphi,\psi}.
    \]
    
    Since the set of simple functions is dense in $L^1$ and is a subset of $QA_{\varphi,\psi}$, $QA_{\varphi,\psi}$ is dense in $L^1$. Let $f\in L^1$, then by density we find $\{f_n\}_{n=1}^\infty\subset QA_{\varphi,\psi}$ such that $f_n\to f$ in $L^1$. Then $\{f_n\}_{n=1}^\infty$ is a Cauchy sequence in $L^1$ and by the above inequality is $\{f_n\}_{n=1}^\infty$ also a Cauchy sequence in $QA_{\varphi,\psi}$. According to Proposition \ref{banach}, $QA_{\varphi,\psi}$ is a quasi-Banach lattice, thus $\{f_n\}_{n=1}^\infty$ is convergent in $QA_{\varphi,\psi}$, let us denote its limit by $g\in QA_{\varphi,\psi}$. Then by (\ref{vnoreni 1}) and triangle inequality,
    \[
    \|f-g\|_1\leq \|f-f_n\|_1+\frac{1}{\psi(1)}\|f_n-g\|_{\varphi,\psi}\to 0,
    \]
    thus $f=g$ and this proves that $QA_{\varphi,\psi}=L^1$.
\end{proof}

\section{Embeddings between the Lorentz spaces and \texorpdfstring{$QA_{\varphi,\psi}$}{QA_varphi,psi}}\label{section 4}

In this section, we describe the connections between Lorentz spaces and the space $QA_{\varphi,\psi}$. As a consequence, we establish the remaining part of the proof of Theorem \ref{vlastnosti}. We begin with the Lorentz spaces into which the space $QA_{\varphi,\psi}$ is embedded.

\begin{thm}\label{incto}
    It holds that $QA_{\varphi,\psi}\hookrightarrow\Lambda_\varphi$ with norm less than or equal to $\frac{1}{\psi(1)}$.
\end{thm}

\begin{proof}
    The following fact was proven in \cite[Proposition 2.2]{MR2898729}: if $\varphi:[0,1]\to[0,\infty)$ is a non-decreasing concave function with $\varphi(0)=0$, then for all $f\in L^\infty$ with $f\not\equiv0$, we have
    \begin{equation}\label{fact}
        \|f\|_{\Lambda_\varphi}\leq \|f\|_\infty\varphi\left(\frac{\|f\|_1}{\|f\|_\infty}\right).
    \end{equation}
    Let $f\in QA_{\varphi,\psi}$ and $\{f_n\}_{n=1}^\infty\subset L^\infty$ be such that $f=\sum_{n=1}^\infty f_n$. Then 
    \[
    \|f\|_{\Lambda_\varphi}\leq \displaystyle\sum_{n=1}^\infty \|f_n\|_{\Lambda_\varphi}\leq\frac{1}{\psi(1)}\displaystyle\sum_{n=1}^\infty \psi(n)\|f_n\|_\infty\varphi\left(\frac{\|f_n\|_1}{\|f_n\|_\infty}\right).
    \]
    This completes the proof.
\end{proof}

    Let $(X,\|\cdot\|)$ be a quasi-normed space whose dual separates the points, then the Mackey norm $\|\cdot\|^c$ on $X$ is defined by 
    \[
    \|x\|^c=\inf\{\alpha>0:\,x\in\alpha\,\mathrm{conv}(B_X)\},\qquad x\in X,
    \]
    where $\mathrm{conv}\,(B_X)$ is the convex hull of the unit ball. The completion of $(X,\|\cdot\|^c)$ is denoted by $\widehat{X}$ and called the Banach envelope of $X$. We shall prove that the Banach envelope of the space $QA_{\varphi,\psi}$ is equal to the Lorenz space $\Lambda_\varphi$.

\begin{thm}\label{envelope}
    We have the equality
    \[
    \widehat{QA_{\varphi,\psi}}=\Lambda_\varphi.
    \]
\end{thm}

\begin{proof}
    Since $QA_{\varphi,\psi}\hookrightarrow \Lambda_\varphi$, and the dual of $\Lambda_\varphi$ separates points, it follows that the dual of $QA_{\varphi,\psi}$ also separates points. The above continuous inclusion gives
    \begin{equation}\label{cnorm}
        \|f\|_{\Lambda_\varphi}\leq \frac{1}{\psi(1)}\|f\|_{\varphi,\psi}^c,\qquad f\in QA_{\varphi,\psi}.
    \end{equation}
    We claim that the converse inequality holds. For every simple function $s=\sum_{j=1}^n a_j\chi_{A_j}$ in the form (\ref{jedoducha}), we have that 
    \begin{align*}
        \|s\|_{\varphi,\psi}^c&\leq \displaystyle\sum_{j=1}^n a_j\|\chi_{A_j}\|_{\varphi,\psi}^c\leq \displaystyle\sum_{j=1}^n a_j\|\chi_{A_j}\|_{\varphi,\psi}\leq \displaystyle\sum_{j=1}^n a_j\psi(1)\|\chi_{A_j}\|_\infty\varphi\left(\frac{\|\chi_{A_j}\|_1}{\|\chi_{A_j}\|_\infty}\right)\\
        &= \psi(1)\displaystyle\sum_{j=1}^n a_j \varphi(\lambda(A_j))=\psi(1)\|s\|_{\Lambda_\varphi}.
    \end{align*}
    By Theorem \ref{vlastnosti}\ref{density} and $\|\cdot\|_{\varphi,\psi}^c\leq \|\cdot\|_{\varphi,\psi}$, the set of simple functions is dense in $(QA_{\varphi,\psi},{\|\cdot\|_{\varphi,\psi}^c})$, therefore for all $\varepsilon>0$ and $f\in QA_{\varphi,\psi}$ we can find a simple function $s$ such that $\|f-s\|_{\varphi,\psi}^c<\varepsilon$. Then by (\ref{cnorm}) also $\psi(1)\|f-s\|_{\Lambda_\varphi}<\varepsilon$, and thus by (\ref{cnorm}) and triangle inequality 
    \[
    \frac{1}{\psi(1)}\|f\|_{\varphi,\psi}^c\leq \frac{1}{\psi(1)}\|s\|_{\varphi,\psi}^c+\frac{\varepsilon}{\psi(1)}\leq \|s\|_{\Lambda_\varphi}+\frac{\varepsilon}{\psi(1)}\leq \|f\|_{\Lambda_\varphi}+\frac{2\varepsilon}{\psi(1)}.
    \]
    Since $\varepsilon$ was arbitrary, we have desired inequality,
    and thus
    \[
    \|f\|_{\varphi,\psi}^c=\psi(1)\|f\|_{\Lambda_\varphi},\qquad f\in QA_{\varphi,\psi}.
    \]
    The set of simple functions is dense in $\Lambda_\varphi$ (see \cite[\Rom{2}.Theorem~5.5]{MR928802}). Since the set of simple functions is a subset of ${QA_{\varphi,\psi}}$, ${QA_{\varphi,\psi}}$ is dense in $\Lambda_\varphi$. Therefore, $\Lambda_\varphi=(\Lambda_\varphi,\psi(1)\|\cdot\|_{\Lambda_\varphi})$ is the desired completion of $(QA_{\varphi,\psi},\|\cdot\|_{\varphi,\psi}^c)$.
\end{proof}

As a consequence of Theorem \ref{envelope}, we obtain the rest of the proof of Theorem \ref{vlastnosti}, that is, ($\Rightarrow$) in \ref{L^infty} and \ref{L^1}.

\begin{proof}[Proof of Theorem \ref*{vlastnosti}, second part]
    \ref*{L^infty} ($\Rightarrow$) Let $QA_{\varphi,\psi}=L^\infty$, then by Theorem \ref{envelope} we have $\Lambda_\varphi=L^\infty$. Therefore 
    \begin{equation*}
        \varphi(t)=\varphi_{\Lambda_\varphi}(t)\asymp\varphi_{L^\infty}(t)=1,\qquad t\in(0,1],
    \end{equation*}
    thus $\varphi(0_+)\neq0$.
    
    \ref*{L^1} ($\Rightarrow$) Similarly as in the previous part, by using Theorem \ref{envelope}, we obtain $\varphi(t)\asymp t$, $t\in(0,1]$.
\end{proof}

From now on, we shall suppose that $\varphi(0_+)=0$ and $\varphi\not\asymp \mathrm{id}$. As a consequence we get that $\varphi$ is continuous on $[0,1]$, $\varphi_+'(0)=\infty$ and there exists a neighborhood of zero where $\varphi$ is strictly concave.
In the next proposition, we construct a family of Lorentz spaces which are embedded into the space $QA_{\varphi,\psi}$.
\begin{prop}\label{prop o vnoreni 0}
    For every sequence $s=\{s_n\}_{n=1}^\infty\in (0,1]^\mathbb{N}$,
    \[
    \Lambda_{\widetilde{\varphi_s}}\hookrightarrow QA_{\varphi,\psi},
    \]
    where $\varphi_s$ is a quasi-concave function on $[0,1]$ defined as 
    \[
    \varphi_s(t)=
        \begin{cases}
            \displaystyle\inf_{n\in\mathbb{N}}\left\{ \max \{s_n,t\}\frac{\varphi(s_n)}{s_n}\psi(n)\right\} & \quad \text{if }\,t\in (0,1],\\
            0 & \quad \text{if }\,t=0.
        \end{cases}
    \]
\end{prop}

\begin{proof}
    Clearly $\varphi_s$ is a quasi-concave function on $[0,1]$. We observe that if $f\in L^\infty$ with $f\not\equiv 0$ and $n\in\mathbb{N}$, then the concavity of $\varphi$ implies
    \begin{align*}
        \|f\|_\infty\varphi\left(\frac{\|f\|_1}{\|f\|_\infty}\right)&= \|f\|_\infty\max\left\{1,\frac{\|f\|_1}{s_n\|f\|_\infty}\right\}\min\left\{1,\frac{s_n\|f\|_\infty}{\|f\|_1}\right\}\varphi\left(\frac{\|f\|_1}{\|f\|_\infty}\right)\\
        &\leq \|f\|_\infty\max\left\{1,\frac{\|f\|_1}{s_n\|f\|_\infty}\right\}\varphi\left(\min\left\{1,\frac{s_n\|f\|_\infty}{\|f\|_1}\right\}\frac{\|f\|_1}{\|f\|_\infty}\right)\\
        &\leq \|f\|_\infty\max\left\{1,\frac{\|f\|_1}{s_n\|f\|_\infty}\right\}\varphi(s_n)\\
        &=\max\{s_n\|f\|_\infty,\|f\|_1\}\frac{\varphi(s_n)}{s_n}.
    \end{align*}
    Observe that for all $a,b,c,d\in\mathbb{R}$, it holds
    \begin{align*}
        \max\{a+b,c+d\}\leq \max\{a,c\}+\max\{b,d\}.
    \end{align*}
    Hence for every simple function $h=\sum_{i=1}^m a_i\chi_{A_i}\in L^\infty$ in the form (\ref{jedoducha}), we have
    \begin{align*}
         \|h\|_{\varphi,\psi}\leq \psi(1)\|h\|_\infty\varphi\left(\frac{\|h\|_1}{\|h\|_\infty}\right)&\leq \psi(1)\max\{s_n\|h\|_\infty,\|h\|_1\}\frac{\varphi(s_n)}{s_n}\\
         &\leq \max\left\{s_n\sum_{i=1}^m\|a_i\chi_{A_i}\|_\infty,\sum_{i=1}^m\|a_i\chi_{A_i}\|_1\right\}\frac{\varphi(s_n)}{s_n}\psi(n)\\
         &\leq \sum_{i=1}^m \max\{s_n\|a_i\chi_{A_i}\|_\infty,\|a_i\chi_{A_i}\|_1\}\frac{\varphi(s_n)}{s_n}\psi(n)\\
         &= \sum_{i=1}^m a_i \max\{s_n,\lambda(A_i)\}\frac{\varphi(s_n)}{s_n}\psi(n).
    \end{align*}
    Therefore 
    \begin{align*}
        \|h\|_{\varphi,\psi}\leq \sum_{i=1}^m a_i\varphi_s(\lambda(A_i))\leq \sum_{i=1}^m a_i\widetilde{\varphi_s}(\lambda(A_i))=\|h\|_{\Lambda_{\widetilde{\varphi_s}}}.
    \end{align*}
    
    Let $f\in\Lambda_{\widetilde{\varphi_s}}$. Since the set of simple functions is dense in $\Lambda_{\widetilde{\varphi_s}}$ (see \cite[\Rom{2}.Theorem~5.5]{MR928802}), we find a sequence of simple functions $\{f_j\}_{j=1}^\infty$ such that $f_j\to f$ in $\Lambda_{\widetilde{\varphi_s}}$. Then $\{f_j\}_{j=1}^\infty$ is a Cauchy sequence in $\Lambda_{\widetilde{\varphi_s}}$ and by the above inequality,$\{f_j\}_{j=1}^\infty$ is also a Cauchy sequence in $QA_{\varphi,\psi}$. According to Proposition \ref{banach}, $QA_{\varphi,\psi}$ is a quasi-Banach lattice, thus $\{f_j\}_{j=1}^\infty$ is convergent in $QA_{\varphi,\psi}$, let us denote its limit by $g\in QA_{\varphi,\psi}$. Then by $\Lambda_{\widetilde{\varphi_s}}\hookrightarrow L^1$ with a norm less than or equal to $K>0$, and by (\ref{vnoreni 1}),
    \[
    \|f-g\|_1\leq K\|f-f_j\|_{\Lambda_{\widetilde{\varphi_s}}}+\frac{1}{\varphi(1)\psi(1)}\|f_j-g\|_{\varphi,\psi}\to 0,
    \]
    thus $f=g$. Then by (\ref{quasi-norm}), we have
    \begin{align*}
        \|f\|_{\varphi,\psi}&\leq 4\|f-f_j\|_{\varphi,\psi}+4\|f_j\|_{\varphi,\psi}\leq 4\|f-f_j\|_{\varphi,\psi}+4\|f_j\|_{\Lambda_{\widetilde{\varphi_s}}}\\
        &\leq 4\|f-f_j\|_{\varphi,\psi}+4\|f_j-f\|_{\Lambda_{\widetilde{\varphi_s}}}+4\|f\|_{\Lambda_{\widetilde{\varphi_s}}}.
    \end{align*}
    By taking limit over both sides, we obtain the required continuous inclusion.
\end{proof}

The previous proposition is not easy to use in applications due to the implicit definition of $\varphi_s$ by using an infimum. The following proposition shows that, under certain conditions on the sequence $s$, the function $\varphi_s$ is equivalent to the function $\alpha_s$, which can be explicitly expressed as a product involving $\varphi$ and $\psi$.
\begin{prop}\label{prop o vnoreni}
    Let $s:[1,\infty)\to(0,1]$ be a decreasing function such that 
    \begin{enumerate}[label=(\roman*)]
        \item $\displaystyle\lim_{x\to\infty} s(x)=0$, 
        \item $\displaystyle\lim_{x\to\infty} \varphi(s(x))\psi(x)=0$, 
        \item there exists $x_0\geq1$ such that $(\varphi\circ s)\psi$ is a non-increasing function on $[x_0,\infty)$,
        \item there exists a constant $C>0$ such that $\frac{\varphi(s(n+1))}{s(n+1)}\leq C\frac{\varphi(s(n))}{s(n)}$ for all $n\in\mathbb{N}$.
    \end{enumerate}
    Then $\Lambda_{\widetilde{\varphi_s}}=\Lambda_{\widetilde{\alpha_s}}$, where 
    \[
    \alpha_s(t)=
        \begin{cases}
            \varphi(t)\psi(s^{-1}(t))& \quad \text{if }\,t\in (0,1],\\
            0 & \quad \text{if }\,t=0.
        \end{cases}
    \]
\end{prop}

\begin{proof}
    Denote $s_n=s(n)$ for all $n\in\mathbb{N}$. Then $s_n\searrow 0$ and $\displaystyle\lim_{n\to\infty}\varphi(s_n)\psi(n)= 0$. For all $t\in[s_{k+1},s_k)$, $k\geq 1$, we have
    \begin{align*}
        \varphi_s(t)&=\displaystyle\inf_{n\geq 1}\max\{{s_n},t\}\frac{\varphi(s_n)}{s_n}\psi(n)\\
        &=\min\left\{\displaystyle\inf_{1\leq n\leq k}\max\{{s_n},t\}\frac{\varphi(s_n)}{s_n}\psi(n),\displaystyle\inf_{n>k}\max\{{s_n},t\}\frac{\varphi(s_n)}{s_n}\psi(n)\right\}\\
        &=\min\left\{\displaystyle\inf_{1\leq n\leq k}s_n\frac{\varphi(s_n)}{s_n}\psi(n),\displaystyle\inf_{n>k}t\frac{\varphi(s_n)}{s_n}\psi(n)\right\}\\
        &=\min\left\{\displaystyle\inf_{1\leq n\leq k}\varphi(s_n)\psi(n),t\frac{\varphi(s_{k+1})}{s_{k+1}}\psi(k+1)\right\}.
    \end{align*}
    Since $\displaystyle\lim_{x\to\infty}\varphi(s(x))\psi(x)= 0$ and $(\varphi\circ s)\psi$ is a non-increasing function on $[x_0,\infty)$, there exists $m_0\in\mathbb{N}$, $m_0\geq x_0$, such that $\{\varphi(s_m)\psi(m)\}_{m=m_0}^\infty$ is non-increasing and $\varphi(s_m)\psi(m)\leq \displaystyle\inf_{1\leq n<m_0}\varphi(s_n)\psi(n)$ for all $m\geq m_0$. We find $C_1>0$ such that for all $m<m_0$, it holds
    \[
    \displaystyle\inf_{1\leq n\leq m}\varphi(s_n)\psi(n)\leq \varphi(s_m)\psi(m)\leq C_1\displaystyle\inf_{1\leq n\leq m}\varphi(s_n)\psi(n).
    \]
    Then either $k\geq m_0$, in which case $\displaystyle\inf_{1\leq n\leq k}\varphi(s_n)\psi(n)=\varphi(s_k)\psi(k)$, or $k<m_0$ and by the above inequality $\displaystyle\inf_{1\leq n\leq k}\varphi(s_n)\psi(n)\asymp \varphi(s_k)\psi(k)$.
    Hence
    \begin{align*}
        \varphi_s(t)&=\min\left\{\displaystyle\inf_{1\leq n\leq k}\varphi(s_n)\psi(n),t\frac{\varphi(s_{k+1})}{s_{k+1}}\psi(k+1)\right\}\\
        &\asymp \min\left\{\varphi(s_k)\psi(k),t\frac{\varphi(s_{k+1})}{s_{k+1}}\psi(k+1)\right\}\\
        &=\min\left\{s_k\frac{\varphi(s_k)}{s_k}\psi(k),t\frac{\varphi(s_{k+1})}{s_{k+1}}\psi(k+1)\right\}.
    \end{align*}
    By subadditivity of $\psi$, we have
    \begin{align*}
        \min\left\{s_k\frac{\varphi(s_k)}{s_k}\psi(k),t\frac{\varphi(s_{k+1})}{s_{k+1}}\psi(k+1)\right\}&\leq t\frac{\varphi(s_{k+1})}{s_{k+1}}\psi(k+1)\leq Ct\frac{\varphi(s_{k})}{s_{k}}\psi(2k)\\
        &\leq 2Ct\frac{\varphi(s_{k})}{s_{k}}\psi(k),
    \end{align*}
    and
    \begin{align*}
        \min\left\{s_k\frac{\varphi(s_k)}{s_k}\psi(k),t\frac{\varphi(s_{k+1})}{s_{k+1}}\psi(k+1)\right\}\geq t\frac{\varphi(s_k)}{s_k}\psi(k).
    \end{align*}
    Therefore
    \[
    \varphi_s(t)\asymp t\frac{\varphi(s_k)}{s_k}\psi(k)\ \asymp t\frac{\varphi(t)}{t}\psi(k)\asymp \varphi(t)\psi(s^{-1}(t)).
    \]
\end{proof}
\begin{cor}\label{cor 1}
    If $\psi\asymp 1$ on $[1,\infty)$, then $QA_{\varphi,\psi}=\Lambda_\varphi$. 
\end{cor}
\begin{proof}
    Without loss of generality, we can assume that $\psi\equiv1$ on $[1,\infty)$. Let us take $s(x)=\frac{1}{x}$, $x\in[1,\infty)$. Then $\displaystyle\lim_{x\to\infty} s(x)=\displaystyle\lim_{x\to\infty}\varphi(s(x))=0$, $\varphi\circ s$ is a non-increasing function on $[1,\infty)$ and for all $n\in \mathbb{N}$, it holds
    \[
    \frac{\varphi(s(n+1))}{s(n+1)}=(n+1)\varphi\left(\frac{1}{n+1}\right)\leq 2n\varphi\left(\frac{1}{n}\right)=2\frac{\varphi(s(n))}{s(n)}.
    \]
    By Proposition \ref{prop o vnoreni}, we get $\Lambda_{\widetilde{\varphi_s}}=\Lambda_\varphi$. Therefore, combining Theorem \ref{incto} and Proposition \ref{prop o vnoreni 0}, we obtain the desired equality.
\end{proof}
   
\section{Comparison to rearrangement-invariant Banach spaces}

Throughout this section, we assume that assumptions (\ref{predpoklad}) and (\ref{predpoklad2}) hold. For completeness, we repeat these assumptions. Assume that there exist constants $c\in(0,1)$, $C\in[1,\infty)$, and $p\in(0,1]$ such that
\begin{equation*}
    \frac{\varphi(t)}{t^c}\ \text{is non-decreasing on $(0,p]$},
\end{equation*}
and for $n\in \mathbb{N}$,
\begin{equation*}
    \psi(n^2)\leq C\psi(n).
\end{equation*}

\begin{prop}\label{thm2}
    For every $N\in\mathbb{N}$, $N\geq2$ there exists a sequence $\{A_j\}_{j=1}^{2N}$ of pairwise disjoint measurable sets in $[0,1]$ such that for $f=\sum_{j=1}^{2N} a_j\chi_{A_j}$ with $a_j=\frac{1}{2N\varphi(\lambda(A_j))}$, we have 
    \[    
        \|f\|_{\varphi,\psi}\geq\frac{2^c-1}{8}\psi(N).
    \]
\end{prop}
    
\begin{proof}
    Define $\gamma(t)=\frac{\varphi(t)}{t}$, $t\in(0,1]$. By concavity of 
    $\varphi$, $\gamma$ is a non-increasing function, satisfying $\gamma(0_+)=\varphi_+'(0)=\infty$. We can find $\mu_1\in(0,\frac{p}{2}]$ such that $\varphi$ is strictly concave on $(0,\mu_1]$. Then $\gamma$ is decreasing on $(0,\mu_1]$, in particular $\gamma$ is invertible. Clearly for $x\geq1$,
    \[
        \gamma\left(\frac{1}{x}\right)= x\varphi\left(\frac{1}{x}\right)\leq x\varphi(1)\leq x\gamma(\mu_1),
    \]
    and so
    \[
        1\geq x\gamma^{-1}(x\gamma(\mu_1)).
    \]
    Thus for $1\leq x_1<x_2$, we have
    \begin{align*}
        x_2\gamma(\mu_1)&=\frac{x_2}{x_1}\gamma(\gamma^{-1}(x_1\gamma(\mu_1)))=\frac{\varphi(\gamma^{-1}(x_1\gamma(\mu_1)))}{\frac{x_1}{x_2}\gamma^{-1}(x_1\gamma(\mu_1)))}\geq \frac{\varphi\left(\frac{x_1}{x_2}\gamma^{-1}(x_1\gamma(\mu_1))\right)}{\frac{x_1}{x_2}\gamma^{-1}(x_1\gamma(\mu_1)))}\\
        &=\gamma\left(\frac{x_1}{x_2}\gamma^{-1}(x_1\gamma(\mu_1))\right),
    \end{align*}
    and therefore
    \[
        x_2\gamma^{-1}(x_2\gamma(\mu_1))\leq x_1\gamma^{-1}(x_1\gamma(\mu_1)),
    \]
    thus $x\gamma^{-1}(x\gamma(\mu_1))$ is non-increasing on $[1,\infty)$, and hence $x\gamma^{-1}(x\gamma(\mu_1))\leq p$, for $x\in[1,\infty)$. Let $N\in\mathbb{N}$, $N\geq 2$ and we define a positive sequence $\{\mu_j\}_{j=1}^{2N}$ as follows:
    \begin{equation}\label{def-posloupnost}
        \mu_{j+1}=\gamma^{-1}\left(\left(N^3\frac{\psi(N)}{\psi(1)}\right)^\frac{1}{c}\gamma(\mu_j)\right).
    \end{equation} 
    Observe that the sequence $\{\mu_j\}_{j=1}^{2N}$ is well defined, because $N^3\psi(N)\geq 2\psi(1)$, and therefore
    \[
    \gamma(\mu_{j+1})=\left(N^3\frac{\psi(N)}{\psi(1)}\right)^\frac{1}{c}\gamma(\mu_j)\geq 2\gamma(\mu_j)= \frac{\varphi(\mu_j)}{\frac{\mu_j}{2}}\geq \frac{\varphi\left(\frac{\mu_j}{2}\right)}{\frac{\mu_j}{2}} = \gamma\left(\frac{\mu_j}{2}\right).
    \]
    Thus
    \begin{equation}\label{posloupnost}
        \begin{aligned}
            \mu_{j+1}&\leq\frac{\mu_j}{2},\\
            \mu_1&\leq\frac{p}{2}\leq \frac{1}{2}.
        \end{aligned}
    \end{equation}
    Let us define 
    \[
    a_j=\frac{1}{2N\varphi(\mu_j)},\qquad 1\leq j\leq2N.
    \]
    Using (\ref{posloupnost}), we find measurable pairwise disjoint sets $A_1,\dots,A_{2N}$ in $[0,1]$ such that $\lambda(A_j)=\mu_j$ for all $1\leq j\leq 2N$.
    We show that $f=\sum_{j=1}^{2N} a_j\chi_{A_j}$ is the desired function. Let $\{f_n\}_{n=1}^\infty\subset L^\infty,\,f_n\geq 0$ be such that $f=\sum_{n=1}^\infty f_n$, then it suffices to prove that
    \[
    \frac{2^c-1}{8}\psi(N)\leq \displaystyle\sum_{n=1}^\infty \psi(n)\|f_n\|_\infty\varphi\left(\frac{\|f_n\|_1}{\|f_n\|_\infty}\right).
    \]
    Now, set $g_N=\sum_{n=N}^\infty f_n$. Since by Theorem \ref{incto} it holds $g_N\in\Lambda_\varphi$, we have two cases.

    Case 1: $\|g_N\|_{\Lambda_\varphi}\geq \frac{2^c-1}{8}$, then by (\ref{fact}),
    \[
    \frac{2^c-1}{8}\psi(N) \leq \psi(N)\|g_N\|_{\Lambda_\varphi} \leq \displaystyle\sum_{n=N}^\infty \psi(n)\|f_n\|_{\Lambda_\varphi} \leq \displaystyle\sum_{n=1}^\infty \psi(n)\|f_n\|_\infty\varphi\left(\frac{\|f_n\|_1}{\|f_n\|_\infty}\right).
    \]

    Case 2: $\|g_N\|_{\Lambda_\varphi}<\frac{2^c-1}{8}$, then for $1\leq j\leq 2N$ and $1\leq l\leq N-1$, we define the sets
    \[
    B_{j,l}=\left\{x\in A_j:\,f_l(x)>\frac{a_j}{2N}\right\}
    \]
    and 
    \[
    B_{j,N}=\left\{x\in A_j:\,g_N(x)>\frac{a_j}{2}\right\}.
    \]
    Then $A_j=\bigcup_{l=1}^N B_{j,l}$, because 
    \begin{align*}
        A_j&\supset\bigcup_{l=1}^N B_{j,l}=\left\{x\in A_j:\,f_1(x)>\frac{a_j}{2N}\lor\dots\lor f_{N-1}(x)>\frac{a_j}{2N}\lor g_N(x)>\frac{a_j}{2} \right\}\\
        &\supset \left\{x\in A_j:\,\sum_{n=1}^\infty f_n(x)> \frac{2N-1}{2N}a_j\right\}= \left\{x\in A_j:\,\sum_{j=1}^{2N} a_j\chi_{A_j}(x)> \frac{2N-1}{2N}a_j\right\}\\
        &=A_j.
    \end{align*}

    Let $J$ be the set of $j$ such that $\lambda(B_{j,N})\geq\frac{\mu_j}{2}$. For all $j\in J$, choose $B_j\subset B_{j,N}$ such that $\lambda(B_j)=\frac{\mu_j}{2}$. Set $g=\sum_{j\in J} \frac{a_j}{2}\chi_{B_j}$. Then, since $a_{j+1}\geq a_j$, we have that, if $J=\{j_k\}_{k=1}^s$, $j_1<\dotsb<j_s$,
    \[
    g^\ast(t)=\displaystyle\sum_{l=1}^s \frac{a_{j_l}}{2}\chi_{[h_{j_l},h_{j_l}+\frac{\mu_{j_l}}{2})}(t),
    \]
    where $h_{j_{l}}=h_{j_{l+1}}+\frac{\mu_{j_{l+1}}}{2}$ and $h_{j_s}=0$.
    Then by (\ref{posloupnost}), 
    \[
    h_{j_l}=\frac{\mu_{j_{l+1}}}{2}+\dotsb+\frac{\mu_{j_s}}{2}\leq \frac{\mu_{j_l}}{4}+\frac{\mu_{j_l}}{8}+\dotsb+\frac{\mu_{j_l}}{2^{s-l}}\leq\frac{\mu_{j_l}}{2}.
    \]
    Since for all $0<h\leq t\leq \frac{p}{4}$, by concavity of $\varphi$ and (\ref{predpoklad}),
    \[
    \varphi(t+h)-\varphi(h)=\frac{\varphi(t+h)-\varphi(h)}{t}t\geq \frac{\varphi(2t)-\varphi(t)}{t}t=(2t)^c\frac{\varphi(2t)}{(2t)^c}-\varphi(t)\geq (2^c-1)\varphi(t),
    \]
    we obtain
    \begin{align*}
        \|g\|_{\Lambda_\varphi}&=\int_0^1 g^\ast(t)\varphi'(t)\mathrm{d}t=\displaystyle\sum_{l=1}^s \frac{a_{j_l}}{2}\left(\varphi\left(h_{j_l}+\frac{\mu_{j_l}}{2}\right)-\varphi(h_{j_l})\right)\\
        &\geq \displaystyle\sum_{l=1}^s \frac{a_{j_l}}{2}(2^c-1)\varphi\left(\frac{\mu_{j_l}}{2}\right)\geq \displaystyle\sum_{l=1}^s \frac{2^c-1}{4} a_{j_l}\varphi(\mu_{j_l})=\frac{s}{N}\frac{2^c-1}{8}.
    \end{align*}
    From here we get, by $g_N\geq g$,
    \[
    \frac{2^c-1}{8}>\|g_N\|_{\Lambda_\varphi}\geq \|g\|_{\Lambda_\varphi}\geq \frac{s}{N}\frac{2^c-1}{8},
    \]
    and consequently $s<N$.

    On the other hand, if $j\in\{1,\dots,2N\}\setminus J$, then $\lambda(\bigcup_{l=1}^{N-1} B_{j,l})>\frac{\mu_j}{2}$ and hence, there exists $l=l(j)$ such that $\lambda(B_{j,l})>\frac{\mu_j}{2N}$.
    Because cardinality of $\{1,\dots,2N\}\setminus J$ is greater than $N$ and there are only $N-1$ $l\mathrm{'s}$, we have that, there exist $j,k\in\{1,\dots,2N\}\setminus J$, $j<k$ for which exists the same $l=l(j)=l(k)$ and
    \[
    \lambda(B_{j,l})>\frac{\mu_j}{2N},\qquad  \lambda(B_{k,l})>\frac{\mu_k}{2N}.
    \]
    Then 
    \[
    f_l\geq f_l\chi_{B_{j,l}}+f_l\chi_{B_{k,l}}> \frac{a_j}{2N}\chi_{B_{j,l}}+\frac{a_k}{2N}\chi_{B_{k,l}}.
    \]
    Therefore by (\ref{usporadani}) and concavity of $\varphi$,
    \begin{align*}
        \|f_l\|_\infty\varphi\left(\frac{\|f_l\|_1}{\|f_l\|_\infty}\right)&\geq \frac{a_k}{2N}\varphi\left(\frac{\frac{a_j\lambda(B_{j,l})+a_k\lambda(B_{k,l})}{2N}}{\frac{a_k}{2N}}\right)
        \geq \frac{1}{4N^2\varphi(\mu_k)}\varphi\left(\frac{a_j\frac{\mu_j}{2N}+a_k\frac{\mu_k}{2N}}{\frac{1}{2N\varphi(\mu_k)}}\right)\\
        &=\frac{1}{4N^2\varphi(\mu_k)}\varphi\left(\frac{\frac{\varphi(\mu_k)}{\varphi(\mu_j)}\mu_j+\mu_k}{2N}\right)
        \geq \frac{1}{8N^3\varphi(\mu_k)}\varphi\left(\frac{\varphi(\mu_k)}{\varphi(\mu_j)}\mu_j\right)\\
        &\geq \frac{2^c-1}{8N^3}\frac{\mu_j}{\varphi(\mu_j)}\gamma\left(\frac{\varphi(\mu_k)}{\varphi(\mu_j)}\mu_j\right)
        \geq \frac{2^c-1}{8N^3}\frac{\mu_j}{\varphi(\mu_j)}\gamma\left(\frac{\varphi(\mu_{j+1})}{\varphi(\mu_j)}\mu_j\right).
    \end{align*}
    Since $x\gamma^{-1}(x\gamma(\mu_1))\leq p$ and also $x\gamma^{-1}(x^j\gamma(\mu_1))\leq p$, for $x\in[1,\infty)$, we obtain by (\ref{predpoklad}),
    \[
        \frac{\varphi(x\gamma^{-1}(x^j\gamma(\mu_1)))}{(x\gamma^{-1}(x^j\gamma(\mu_1)))^c}\geq \frac{\varphi(\gamma^{-1}(x^j\gamma(\mu_1)))}{(\gamma^{-1}(x^j\gamma(\mu_1)))^c},
    \]
    therefore
    \[
        \varphi(x\gamma^{-1}(x^j\gamma(\mu_1)))\geq x^c\varphi(\gamma^{-1}(x^j\gamma(\mu_1))),
    \]
    and so
    \[
        \gamma(x\gamma^{-1}(x^j\gamma(\mu_1)))\geq x^{c-1}\gamma(\gamma^{-1}(x^j\gamma(\mu_1)))=x^{c+j-1}\gamma(\mu_1),
    \]
    which implies
    \[
        x\gamma^{-1}(x^j\gamma(\mu_1)))\leq \gamma^{-1}(x^{c+j-1}\gamma(\mu_1)).
    \]
    By choosing $x=\left(N^3\frac{\psi(N)}{\psi(1)}\right)^\frac{1}{c}$ and using (\ref{def-posloupnost}), we get
    \begin{align*}
        \left(N^3\frac{\psi(N)}{\psi(1)}\right)^\frac{1}{c}\gamma^{-1}\left(\left(N^3\frac{\psi(N)}{\psi(1)}\right)^\frac{j}{c}\gamma(\mu_1))\right)&\leq \gamma^{-1}\left(\left(N^3\frac{\psi(N)}{\psi(1)}\right)^{1+\frac{j-1}{c}}\gamma(\mu_1)\right),\\
        \left(N^3\frac{\psi(N)}{\psi(1)}\right)^\frac{1}{c}\gamma^{-1}\left(\left(N^3\frac{\psi(N)}{\psi(1)}\right)^\frac{j-1}{c}\gamma(\mu_2)\right)&\leq \gamma^{-1}\left(\left(N^3\frac{\psi(N)}{\psi(1)}\right)^{1+\frac{j-2}{c}}\gamma(\mu_2)\right),\\
        &\vdots\\
        \left(N^3\frac{\psi(N)}{\psi(1)}\right)^\frac{1}{c}\gamma^{-1}\left(\left(N^3\frac{\psi(N)}{\psi(1)}\right)^\frac{1}{c}\gamma(\mu_j)\right)&\leq \gamma^{-1}\left(N^3\frac{\psi(N)}{\psi(1)}\gamma(\mu_j)\right),\\
        \left(N^3\frac{\psi(N)}{\psi(1)}\right)^\frac{1}{c}\mu_{j+1}&\leq \gamma^{-1}\left(N^3\frac{\psi(N)}{\psi(1)}\gamma(\mu_j)\right),\\
        \gamma\left(\left(N^3\frac{\psi(N)}{\psi(1)}\right)^\frac{1}{c}\mu_{j+1}\right)&\geq N^3\frac{\psi(N)}{\psi(1)}\gamma(\mu_j),\\
        \gamma\left(\frac{\varphi(\mu_{j+1})}{\gamma(\mu_j)}\right)&\geq N^3\frac{\psi(N)}{\psi(1)}\gamma(\mu_j).
    \end{align*}
    Therefore
    \[
    \|f_l\|_\infty\varphi\left(\frac{\|f_l\|_1}{\|f_l\|_\infty}\right)\geq \frac{2^c-1}{8N^3}\frac{\mu_j}{\varphi(\mu_j)}\gamma\left(\frac{\varphi(\mu_{j+1})}{\varphi(\mu_j)}\mu_j\right)\geq \frac{2^c-1}{8}\frac{\psi(N)}{\psi(1)},
    \]
    and thus 
    \[
    \frac{2^c-1}{8}\psi(N)\leq  \psi(1)\|f_l\|_\infty\varphi\left(\frac{\|f_l\|_1}{\|f_l\|_\infty}\right)\leq \displaystyle\sum_{n=1}^\infty \psi(n)\|f_n\|_\infty\varphi\left(\frac{\|f_n\|_1}{\|f_n\|_\infty}\right).
    \]
\end{proof}

\begin{prop}\label{thm3}
    Let X be an r.i.\ Banach space such that $X\hookrightarrow QA_{\varphi,\psi}$ with norm less than or equal to $K$. Let $\{\omega_N(X)\}_{N=1}^\infty$ be the sequence defined by $\omega_N(X)=\displaystyle\sup_{1\leq j\leq 2N}\frac{\varphi_X(\mu_j)}{\varphi(\mu_j)}$, where $\{\mu_j\}_{j=1}^{2N}$ is defined similarly as in \eqref{def-posloupnost}, but with $\mu_1$ less than or equal to the corresponding value in (\ref{def-posloupnost}). Then 
    \[
    \displaystyle\inf_{N\geq2} \frac{\omega_N(X)}{\psi(N)}\geq \frac{2^c-1}{8K}.
    \]
\end{prop}

\begin{proof}
    Let $f=\sum_{j=1}^{2N}a_j\chi_{A_j}$ be the function constructed in the proof of Proposition \ref{thm2}. Since $f$ is a simple function, $f$ is contained in $X$. From Proposition \ref{thm2}, we have that for all $N\in\mathbb{N}$, $N\geq2$,
    \begin{align*}
        \frac{2^c-1}{8}\psi(N)&\leq \left\|\displaystyle\sum_{j=1}^{2N} a_j\chi_{A_j}\right\|_{\varphi,\psi}\leq K\left\|\displaystyle\sum_{j=1}^{2N} a_j\chi_{A_j}\right\|_X\leq K\displaystyle\sum_{j=1}^{2N} a_j\|\chi_{A_j}\|_X\\
        &=K\displaystyle\sum_{j=1}^{2N} \frac{\|\chi_{A_j}\|_X}{2N\varphi(\mu_j)}=\frac{K}{2N}\displaystyle\sum_{j=1}^{2N} \frac{\varphi_X(\mu_j)}{\varphi(\mu_j)}\leq K\omega_N(X),
    \end{align*}
    which completes the proof.
\end{proof}

Let us now recall the definition of the function $\tau$. We define 
\begin{equation*}
    \tau(t)=\begin{cases}
            \varphi(t)\psi\left(\overline{\log}\,\left(\frac{\varphi(t)}{t}\right)\right)& \quad \text{if }\,t\in (0,1],\\
            0 & \quad \text{if }\,t=0,
        \end{cases}
\end{equation*}
where $\overline{\log}\,t=1+\log^+t$, $t>0$.
Using the previous proposition, we are now prepared to prove our second main result, Theorem \ref{trida neobs}.

    \begin{proof}[Proof of Theorem \ref*{trida neobs}]
        \ref*{trida neobss}
        Let us define $\gamma(t)=\frac{\varphi(t)}{t}$, $t\in(0,1]$. We write $\varphi_X(t)=\varphi(t)\phi(t)$ for all $t\in[0,1]$, where
        \[
        \displaystyle\lim_{t\to 0_+}\frac{\phi(t)}{\psi(\overline{\log}\,\gamma(t))}=0.
        \]
        Given $\varepsilon>0$, we find $\delta>0$ such that $\phi(t)<\varepsilon\psi(\overline{\log}\,\gamma(t))$, for all $t\in(0,\delta]$. Pick $N\in\mathbb{N}$, $N\geq 2$ and the sequence $\{\mu_j\}_{j=1}^{2N}$ constructed in the proof of Proposition \ref{thm2} with $\mu_1\leq\delta$. Then we have 
        \[
        \phi(\mu_j)<\varepsilon\psi(\overline{\log}\,\gamma(\mu_j)),\quad 1\leq j\leq 2N.
        \]
        By the definition of $\mu_j$, we have
        \begin{align*}
            \overline{\log}\,\gamma(\mu_j)&\leq \frac{1}{c}\overline{\log}\,\left(N^3\frac{\psi(N)}{\psi(1)}\right)+\overline{\log}\,\gamma(\mu_{j-1})\leq \frac{4}{c}\overline{\log}\,N+\overline{\log}\,\gamma(\mu_{j-1})\\
            &\leq\dotsb\leq \frac{4(j-1)}{c}\overline{\log}\,N+\overline{\log}\,\gamma(\delta)
            \leq \frac{8N}{c}\overline{\log}\,N+\overline{\log}\,\gamma(\delta)\leq N^2
        \end{align*}
        for big enough $N$. Hence, using (\ref{predpoklad2}),
        \[
        \phi(\mu_j)<\varepsilon\psi(\overline{\log}\,\gamma(\mu_j))\leq \varepsilon\psi(N^2)\leq C\varepsilon\psi(N).
        \]
        It then follows that
        \[
        \omega_N(X)=\displaystyle\sup_{1\leq j\leq 2N}\frac{\varphi_X(\mu_j)}{\varphi(\mu_j)}=\displaystyle\sup_{1\leq j\leq 2N}\phi(\mu_j)\leq C\varepsilon\psi(N).
        \]
        Thus
        \[
        \displaystyle\inf_{N\geq 2}\frac{\omega_N(X)}{\psi(N)}\leq C\varepsilon.
        \]
        Therefore Proposition \ref{thm3} implies that $QA_{\varphi,\psi}$ does not contain $X$.

    \ref*{vnor}
    Let $\gamma$ and $\mu_1$ be the same as in the proof of Proposition \ref{thm2}. Since $\gamma(0_+)=\infty$, we find $\delta\in(0,\min\{t_0,\mu_1\}]$ such that $\gamma(t)\geq1$, for $t\in(0,\delta]$. Let us define $s(x)=\gamma^{-1}(e^{x-1})$, then $s$ is a decreasing function on $[\overline{\log}\,\gamma(\delta),\infty)$. It is sufficient to show that $s$ satisfies the conditions of Proposition \ref{prop o vnoreni} on $[\overline{\log}\,\gamma(\delta),\infty)$, because then $\tau= \alpha_s\asymp\varphi_s$ on $[0,\delta]$, and therefore $\tau\asymp \varphi_s$ on $[0,1]$. 
    
    Clearly $\displaystyle\lim_{x\to\infty}s(x)=0$ and $\tau\circ s=(\varphi\circ s)\psi$ is a non-increasing function on $[\overline{\log}\,\gamma(\delta),\infty)$. By the definition of $\gamma$, we have on $[\gamma(\delta),\infty)$,
    \[
        \gamma\left(\frac{\varphi(1)}{x}\right)\leq \frac{x}{\varphi(1)}\varphi\left(\frac{\varphi(1)}{x}\right)\leq x,
    \]
    and so
    \[
        \frac{\varphi(1)}{x}\geq \gamma^{-1}(x).
    \]
    Since $\varphi(1)e^{1-x}(x+1)^\frac{1}{c}\to 0$, for $x\to \infty$, we find $x_0\geq \overline{\log}\,\gamma(\delta)$ such that $\varphi(1)e^{1-x}(x+1)^\frac{1}{c}\leq p$, for all $x\geq x_0$. Hence, by subadditivity of $\psi$ and (\ref{predpoklad}), we have for $x\geq x_0$,
    \begin{align*}
        \varphi(\gamma^{-1}(e^{x-1}))\psi(x)&\leq \varphi(\varphi(1)e^{1-x})\psi([x]+1)\leq \varphi(\varphi(1)e^{1-x})([x]+1)\psi(1)\\
        &\leq \varphi(\varphi(1)e^{1-x})((x+1)^{\frac{1}{c}})^c\psi(1)\\
        &\leq \varphi(\varphi(1)e^{1-x}(x+1)^{\frac{1}{c}})\psi(1)\to 0,
    \end{align*}
     where $[x]$ is the floor function. Thus, $\displaystyle\lim_{x\to\infty}\varphi(s(x))\psi(x)=0$. Clearly for $n\in\mathbb{N}$, $n\geq\overline{\log}\,\gamma(\delta)$,
     \[
     \frac{\varphi(s(n+1))}{s(n+1)}=\gamma(\gamma^{-1}(e^{n}))=ee^{n-1}=e\gamma(\gamma^{-1}(e^{n-1}))=e\frac{\varphi(s(n))}{s(n)},
     \]
    which completes the proof.

    \ref*{char}
    One implication follows from Corollary \ref{cor 1}. Let us suppose that $QA_{\varphi,\psi}=\Lambda_\varphi$. Then, by Theorem \ref{trida neobs}\ref{trida neobss}, we obtain 
    \[
    \lim_{t\to 0_+}\frac{1}{\psi\left(\overline{\log}\,\left(\frac{\varphi(t)}{t}\right)\right)}\not=0.
    \]
    Since
    \[
    \lim_{t\to0_+}\overline{\log}\,\left(\frac{\varphi(t)}{t}\right)=\infty,
    \]
    we have $\psi\asymp1$ on $[1,\infty).$
\end{proof}

We finish this paper by providing examples of the function $\tau$ for various choices of $\varphi$ and $\psi$.

\begin{exampl}
    For $\alpha\in(0,1]$ and $\beta,\gamma \in[0,1]$, let us define
    \[
    \varphi(t)=\begin{cases}
            t^\alpha\left(\overline{\log}\,\frac{1}{t}\right)^\beta & \quad \text{if }\,t\in (0,\min\{1,e^{1-\frac{\beta}{\alpha}}\}],\\
            e^{\alpha-\beta}\left(\frac{\beta}{\alpha}\right)^\beta &\quad \text{if }\,t\in(\min\{1,e^{1-\frac{\beta}{\alpha}}\},1],\\
            0 & \quad \text{if }\,t=0,
        \end{cases}
    \]
    and
    \[
    \psi(t)=\begin{cases}
            (\overline{\log}\,t)^\gamma& \quad \text{if }\,t\in [1,\infty),\\
            t & \quad \text{if }\,t\in[0,1).
        \end{cases}
    \]
    We recall that the choice $\alpha=\beta=\gamma=1$ corresponds to the Arias-de-Reyna space $QA$.
    The choice $\alpha=1$ and $\beta=0$ gives the space $L^1$, and therefore we do not consider it here.
    Thus, on the interval $[0,1]$, we have
        \[
    \tau(t)\asymp\begin{cases}
        t^{\alpha}\left(\overline{\log}\,\frac{1}{t}\right)^{\beta}\left(\overline{\log}\,\overline{\log}\,\frac{1}{t}\right)^\gamma &\quad \text{if }\,\alpha<1,\\[0.5ex]
        t\left(\overline{\log}\,\frac{1}{t}\right)^{\beta}\left(\overline{\log}\,\overline{\log}\,\overline{\log}\,\frac{1}{t}\right)^\gamma &\quad \text{if }\,\alpha=1,\, \beta\not=0.
    \end{cases}
    \]
    \end{exampl}

\section*{Acknowledgments}
The author would like to thank Lenka Slav\'ikov\'a for many valuable comments and suggestions.

\bibliography{bibliography}{}
\bibliographystyle{plain}

\end{document}